\documentclass[10pt]{article}
\font\smallit=cmti10 \font\smalltt=cmtt10 
\usepackage{amssymb,latexsym,amsmath,epsfig} 

\makeatletter

\renewcommand\section{\@startsection {section}{1}{\z@}
	{-30pt \@plus -1ex \@minus -.2ex} {2.3ex \@plus.2ex}
	{\normalfont\normalsize\bfseries}}

\renewcommand\subsection{\@startsection{subsection}{2}{\z@}
	{-3.25ex\@plus -1ex \@minus -.2ex} {1.5ex \@plus .2ex}
	{\normalfont\normalsize\bfseries}}

\renewcommand{\@seccntformat}[1]{\csname the#1\endcsname. }
\newtheorem{theorem}{Theorem}[section]
\newtheorem{lemma}[theorem]{Lemma}
\newtheorem{corollary}[theorem]{Corollary}

\newtheorem{definition}{Definition}
\newenvironment{proof}
{\vskip 0.15in \par\noindent{\emph{Proof}.}\hskip 0.5em\ignorespaces}
{\hfill $\Box$\par\medskip}
\makeatother

\begin{document}

\begin{center}
	{\bf A NOTE ON THE NUMBER OF REPRESENTATIONS OF $n$ AS A SUM OF GENERALIZED POLYGONAL NUMBERS} \vskip 20pt
	{\bf Subhajit Bandyopadhyay}\\
	{\smallit Department of Mathematical Sciences, Tezpur University, Napaam-784028, Sonitpur, Assam, India}\\
	{\tt subha@tezu.ernet.in}\\  \vskip 10pt
	{\bf Nayandeep Deka Baruah}\\
	{\smallit Department of Mathematical Sciences, Tezpur University, Napaam-784028, Sonitpur, Assam, India}\\
	{\tt nayan@tezu.ernet.in}\\
\end{center}
\vskip 30pt

\centerline{\textbf{Abstract}}
Recently, Jha \cite{Jha1, Jha2} has found identities that connect certain sums over the divisors of $n$ to the number of representations of $n$ as a sum of squares and triangular numbers. In this note, we state a generalized result that gives such relations for $s$-gonal numbers for any integer $s\geq3$.

\pagestyle{myheadings}

\markright{\smalltt INTEGERS: (2022)\hfill}

\thispagestyle{empty}

\baselineskip=12.875pt

\vskip 30pt

\vspace*{-\baselineskip}

\small

\section{\textbf{Introduction}}
Jha \cite{Jha1, Jha2} has obtained two identities that connect certain sums over the divisors of $n$ to the number of representations of $n$ as sums of squares and sums of triangular numbers, respectively. Our objective is to show that these results can be generalized to the number of representations of $n$ as a sum of any specific generalized polygonal number. We also obtain some corollaries, including Jha's results \eqref{corEqMain} and \eqref{triJha}. 
In this section, we introduce two definitions and related tools to use in the later sections.
\begin{definition}
For an integer $s\ge3$, the generalized $n^{\text{th}}$ $s$-gonal number is defined by
$$F_s(n) := \frac{(s-2)n^2 - (s-4)n}{2}, \quad n\in \mathbb{Z}. $$
Henceforth, we call these numbers as $s$-gonal numbers.
\end{definition}

\noindent
The generating function $G_s(q)$ of $F_s(n)$ is given by
$$G_s(q) := \sum_{n=-\infty}^{\infty} q^{F_s(n)} = f(q, q^{s-3}),$$
where $f(a,b)$ is the Ramanujan's theta function defined by \cite[p. 34]{Berndt}:
\begin{align*}
f(a,b)=\sum_{n=-\infty}^\infty a^{n(n+1)/2}b^{n(n-1)/2},\quad |ab|<1.
\end{align*}
Also note the exceptional case that $G_3(q)$ generates each triangular number twice while $G_6(q)$ generates only once.

\begin{definition} \textup{(Comtet, \cite[p. 133]{Comtet})}
The partial Bell polynomials are the polynomials $B_{n,k} \equiv B_{n,k}(x_{1},x_{2},\cdots,x_{n-k+1})$ in an infinite number of variables defined by the formal double series expansion:
\begin{align*}
\sum_{n,k \geq 0} B_{n,k} \frac{t^n}{n!} u^k =  \textup{exp}\left(u \sum_{m\geq1} x_m \frac{t^m}{m!} \right).   
\end{align*}
\end{definition}
For more equivalent definitions, exact expressions and further results involving the Bell polynomials, we refer to \cite[Chap. 3.3]{Comtet}.

The next section contains some lemmas and the main theorem. In the final section, we present  some corollaries, including Jha's results \eqref{corEqMain} and \eqref{triJha}.

\section{\textbf{Lemmas and Main Theorem \ref{theorem1}}}
\noindent
In the following, we state and prove two lemmas that lead us to the main theorem, i.e., Theorem \ref{theorem1}.
\noindent
\begin{lemma} \label{lem1}
Let $n$ be a positive integer. Then we have the following result.
\begin{equation*}\begin{aligned}
\sum_{d|n} &\dfrac{1}{d}\left((-1)^d \delta_1\left(\frac{n}{d},s-2\right) + \delta_2\left(\frac{n}{d},s-2\right)\right) \notag\\
&= \frac{1}{n!} \sum_{k=1}^{n} (-1)^k (k-1)! B_{n,k}(G'_{s}(0),G''_{s}(0),\ldots,G^{n-k+1}_{s}(0)),
\end{aligned}
\end{equation*}
where, we define $\delta_1(m,v)$ and $\delta_2(m,v)$ for $v \geq 2$ as follows: 
\begin{equation*}
\delta_1(m,v) = 
\begin{cases}
2, &\quad\text{if } m \equiv 1 (\textup{mod}~2),~ v=2, \\
1, &\quad\text{if } m \equiv 1 \text{ or } (v-1) (\textup{mod}~v),~ v \geq 3, \\
0, &\quad\text{otherwise},
\end{cases}    
\end{equation*}
and
\begin{equation*}
\delta_2(m,v) = 
\begin{cases}
1, &\quad\text{if } m \equiv 0 (\textup{mod}~v),~ v \geq 2, \\
0, &\quad\text{otherwise.}
\end{cases}    
\end{equation*}
\end{lemma}
\begin{proof}
By Jacobi triple product identity \cite[p. 35, Entry 19]{Berndt}, we have
\begin{align*}
G_s(q) &= f(q,q^{s-3}) \\
&= (-q;q^{s-2})_{\infty}(-q^{s-3};q^{s-2})_{\infty}(q^{s-2};q^{s-2})_{\infty} \\
&= \prod_{j=0}^{\infty} \left((1+q^{(s-2)j+1})(1+q^{(s-2)j+s-3})(1-q^{(s-2)(j+1)})\right).
\end{align*}
Therefore,
\begin{equation}\label{lhslem1}
\begin{aligned}\log G_s(q)&= \sum_{j=0}^{\infty}\left(\log(1+q^{(s-2)j+1}) + \log(1+q^{(s-2)j+s-3}) + \log(1-q^{(s-2)(j+1)}) \right) \\
&= - \sum_{j=0}^{\infty} \sum_{\ell=1}^{\infty} \left(\frac{(-1)^\ell}{\ell} q^{((s-2)j+1)\ell} + \frac{(-1)^\ell}{\ell} q^{((s-2)j+s-3)\ell} + \frac{1}{\ell} q^{(s-2)(j+1)\ell}  \right) \\
&= - \sum_{\substack{j\geq1 \\ j\equiv1~(\textup{mod}~{s-2})}} \sum_{\ell\geq1} \frac{(-1)^\ell}{\ell} q^{j\ell} - \sum_{\substack{j\geq1 \\ j\equiv-1~(\textup{mod}~{s-2})}} \sum_{l\geq1} \frac{(-1)^\ell}{\ell} q^{j\ell} \\
&\quad- \sum_{\substack{j\geq1 \\ j\equiv0~(\textup{mod}~{s-2})}} \sum_{\ell\geq1} \frac{1}{\ell} q^{j\ell} \\
&= - \sum_{n\geq1} q^{n} \left( \sum_{d|n} \dfrac{1}{d}\left((-1)^d \delta_1\left(\frac{n}{d},s-2\right) + \delta_2\left(\frac{n}{d},s-2\right)\right) \right), 
\end{aligned}
\end{equation}
where the given definitions of $\delta_1(m,v)$ and $\delta_2(m,v)$ follow naturally.

\noindent
Now, let the Taylor series expansion of $G_s(q)$ be 
$$G_s(q)= \sum_{n\geq0}g_n \frac{q^n}{n!}.$$Then, from \cite[p. 140, (5a) and (5b)]{Comtet}, we have the following result:
\begin{equation} \label{rhslem1}
\log G_s(q) = \sum_{n\geq1} L_n \frac{q^n}{n!},
\end{equation}
where
$$L_n = L_n(g_1, g_2,\ldots, g_n) = \sum_{k=1}^{n}(-1)^k (k-1)! B_{n,k}(g_1, g_2,\ldots, g_n).$$
Comparing \eqref{lhslem1} and \eqref{rhslem1}, we arrive at the desired result.
\end{proof}

\noindent
\begin{lemma} \label{lem2}
Let $t_{s,j}(n)$ denote the number of representations of $n$ as a sum of $j$ $s$-gonal numbers. Then, we have
$$B_{n,k}(G'_{s}(0),G''_{s}(0),\cdots,G^{n-k+1}_{s}(0)) = \frac{n!}{k!} \sum_{j=1}^{k} (-1)^{k-j} \binom{k}{j} t_{s,j}(n) .$$
\end{lemma}
\begin{proof}
The proof is similar to the one given in \cite[Lemma 2]{Jha2}. So we omit.
\end{proof}

\noindent
\begin{theorem} \label{theorem1}
For all positive integers $n,s$ with $s\geq4$, we have  
$$\sum_{d|n} \frac{1}{d}\left((-1)^d \delta_1\left(\frac{n}{d},s-2\right) + \delta_2\left(\frac{n}{d},s-2\right)\right) = \sum_{j=1}^{n} \frac{(-1)^j}{j} \binom{n}{j} t_{s,j}(n).$$
\end{theorem}
\begin{proof}
Our proof of the theorem is essentially similar to the one given in \cite[Theorem 1]{Jha2}. From Lemma \ref{lem1} and Lemma \ref{lem2}, we have
\begin{align*}
&\sum_{d|n} \frac{1}{d}\left((-1)^d \delta_1(\frac{n}{d},s-2) + \delta_2(\frac{n}{d},s-2)\right) \\
&= \frac{1}{n!} \sum_{k=1}^{n} (-1)^k (k-1)! \frac{n!}{k!} \sum_{j=1}^{k} (-1)^{k-j} \binom{k}{j} t_{s,j}(n) \\
&= \sum_{k=1}^{n} \sum_{j=1}^{k} \frac{(-1)^{j}}{k} \binom{k}{j} t_{s,j}(n) \\ 
&= \sum_{j=1}^{n} (-1)^{j} t_{s,j}(n) \sum_{k=j}^{n} \frac{1}{k} \binom{k}{j} \\
&= \sum_{j=1}^{n} (-1)^{j} \frac{1}{j} \binom{n}{j} t_{s,j}(n), 
\end{align*}
where we have used the result 
$$\sum_{k=j}^{n} \frac{1}{k} \binom{k}{j} = \frac{1}{j} \binom{n}{j} $$
that can be derived easily from the identity
$$\binom{k}{j-1} = \binom{k+1}{j} - \binom{k}{j}.$$
\end{proof}

\section{\textbf{Corollaries}}
\noindent
In this section, we present  the results in \cite{Jha1,Jha2} and some other interesting results as  corollaries to Main Theorem \ref{theorem1}.

\begin{corollary}\textup{(Jha \cite[Theorem 1]{Jha2})}
For any positive integer $n$, we have
\begin{align}\label{corEqMain}\sum_{\substack{d|n \\ d~\textup{odd}}} \frac{2(-1)^n}{d} = \sum_{j=1}^{n} \frac{(-1)^j}{j} \binom{n}{j} t_{4,j}(n).\end{align}
\end{corollary}
\begin{proof}
Setting $s=4$ in Theorem \ref{theorem1}, we find that
\begin{align}\label{corEq1}
\sum_{\substack{d|n \\ \frac{n}{d} \text{ odd}}} \frac{2(-1)^d}{d} + \sum_{\substack{d|n \\ \frac{n}{d} \text{ even}}} \frac{1}{d} = \sum_{j=1}^{n} \frac{(-1)^j}{j} \binom{n}{j} t_{4,j}(n).
\end{align}
To prove that \eqref{corEq1} is equivalent to \eqref{corEqMain}, it is enough to show that 
\begin{align}\label{corEq2}\sum_{\substack{d|n \\ d~\textup{odd}}} \frac{2(-1)^n}{d} = \sum_{\substack{d|n \\ \frac{n}{d} \text{ odd}}} \frac{2(-1)^d}{d} + \sum_{\substack{d|n \\ \frac{n}{d} \text{ even}}} \frac{1}{d} .\end{align}
We complete it by considering the following three possible cases of $n$.

\noindent \emph{Case I, $n$ is odd}: In this case, it is easily seen that both sides of \eqref{corEq2} become 
$$-\sum_{\substack{d|n \\ d~\textup{odd}}} \frac{2}{d}.$$

\noindent \emph{Case II, $n=2^k$, where $k\geq1$}: In this case, the left-hand side of \eqref{corEq2} is equal to $2$.

Now, the right-hand side of \eqref{corEq2} becomes
\begin{align*}\sum_{\substack{d|2^k \\ \frac{2^k}{d} \text{ odd}}} \frac{2(-1)^d}{d} + \sum_{\substack{d|2^k \\ \frac{2^k}{d} \text{ even}}} \frac{1}{d}&=\frac{1}{2^{k-1}}+\left(1+\frac{1}{2}+\frac{1}{2^2}+\cdots+\frac{1}{2^{k-1}}\right)\\
&=\frac{1}{2^{k-1}}+\frac{2^k-1}{2^{k-1}}\\
&=2.\end{align*}
Thus, \eqref{corEq2} holds good for this case.

\noindent \emph{Case III, $n=2^km$ where $k\geq1$ and $m$ is odd and greater than 1}: The left-hand side of \eqref{corEq2} becomes
\begin{align}\label{corEq3}\sum_{\substack{d|n \\ d~\textup{odd}}} \frac{2}{d}.
\end{align}

\noindent
Again, the right-hand side of \eqref{corEq2} is
\begin{equation}\label{corEq4}
\begin{aligned}&\sum_{\substack{d|2^km \\ \frac{2^km}{d} \text{ odd}}} \frac{2(-1)^d}{d} + \sum_{\substack{d|2^km \\ \frac{2^km}{d} \text{ even}}} \frac{1}{d}\\
&=\sum_{\substack{d|2^km \\ d~ \textup{odd}}} \frac{2}{2^kd} + \Big(\sum_{\substack{d|2^km \\ d~ \text{odd}}} \frac{1}{d} +\sum_{\substack{d|2^km \\ d~ \text{odd}}} \frac{1}{2d}+\sum_{\substack{d|2^km \\ d~ \text{odd}}} \frac{1}{2^2d}+\cdots+\sum_{\substack{d|2^km \\ d~ \text{odd}}} \frac{1}{2^{k-1}d}\Big)\\
&=\left(\frac{1}{2^{k-1}} + 1+\frac{1}{2}+\frac{1}{2^2}+\cdots+ \frac{1}{2^{k-1}}\right)\sum_{\substack{d|n \\ d~ \text{odd}}} \frac{1}{d}\\
&=\sum_{\substack{d|n \\ d~\textup{odd}}} \frac{2}{d}.\end{aligned}
\end{equation}
From \eqref{corEq3} and \eqref{corEq4}, we conclude that \eqref{corEq2} holds good for this case as well. 
\end{proof}

\noindent
\begin{corollary}\textup{(Jha \cite[Theorem 1]{Jha1})} 
For any positive integer $n$, we have
\begin{align}\label{triJha}\sum_{d|n}\frac{1+2(-1)^d}{d}= \sum_{j=1}^n \frac{(-1)^j}{j} \binom{n}{j} t_{6,j}(n).
\end{align}
\end{corollary}
\begin{proof}
Setting $s=6$ in Theorem \ref{theorem1}, we obtain
\begin{align}\label{trinew}\sum_{\substack{d|n \\ \frac{n}{d}~\textup{odd}}}\frac{(-1)^d}{d}+\sum_{\substack{d|n \\ \frac{n}{d}\equiv0~(\textup{mod}~4)}}\frac{1}{d}= \sum_{j=1}^n \frac{(-1)^j}{j} \binom{n}{j} t_{6,j}(n).
\end{align}
To prove the equivalence of \eqref{trinew} and \eqref{triJha}, it is enough to show that
\begin{align}\label{triequiv}\sum_{\substack{d|n \\ \frac{n}{d}~\textup{odd}}}\frac{(-1)^d}{d}+\sum_{\substack{d|n \\ \frac{n}{d}\equiv0~(\textup{mod}~4)}}\frac{1}{d}= \sum_{d|n}\frac{1+2(-1)^d}{d}.
\end{align}
We show it by considering three possible cases of $n$.

\noindent \emph{Case I, $n$ is odd}: In this case, we notice that both sides of \eqref{triequiv} become
$$-\sum_{\substack{d|n \\ d~\textup{odd}}}\frac{1}{d}.$$

\noindent \emph{Case II, $n=2m$ with $m$ odd}: In this case, the left-hand side of \eqref{triequiv} is
\begin{equation*}\begin{aligned}\sum_{\substack{d|2m\\ \frac{2m}{d} \text{ odd}}} \frac{(-1)^d}{d} + \sum_{\substack{d|2m \\ \frac{2m}{d}\equiv0~(\textup{ mod}~4)}} \frac{1}{d}&=\sum_{\substack{2d|2m\\ d~ \text{odd}}} \frac{(-1)^{2d}}{2d}\notag\\
&=\frac{1}{2}\sum_{\substack{d|2m\\ d ~\text{odd}}} \frac{1}{d}.
\end{aligned}
\end{equation*}

\noindent 
The right-hand side of \eqref{triequiv} is
\begin{equation*}\begin{aligned}\sum_{\substack{d|2m\\ d ~\text{odd}}} \frac{1+2(-1)^{d}}{d} + \sum_{\substack{d|2m\\ d ~\text{even}}} \frac{1+2(-1)^{d}}{d}&=-\sum_{\substack{d|2m\\ d ~\text{odd}}} \frac{1}{d}+\frac{3}{2}\sum_{\substack{d|2m\\ d ~\text{odd}}} \frac{1}{d}\notag\\
&=\frac{1}{2}\sum_{\substack{d|2m\\ d ~\text{odd}}} \frac{1}{d}.\end{aligned}
\end{equation*}
Thus, \eqref{triequiv} holds good in this case. 

\noindent \emph{Case III, $n=2^km$ with $m$ odd and $k\geq2$}: In this case, the left-hand side of \eqref{triequiv} is

\begin{equation}\label{triequivleft}\begin{aligned}&\sum_{\substack{d|2^km\\ \frac{2^km}{d} ~\text{odd}}} \frac{(-1)^{d}}{d} +\sum_{\substack{d|2^km\\ \frac{2^km}{d}\equiv0~(\textup{mod}~4)}} \frac{1}{d}\\
&=\frac{1}{2^k}\sum_{\substack{d|2^km\\ d ~\text{odd}}} \frac{1}{d}+\left(1+\frac{1}{2}+\frac{1}{2^2}+\frac{1}{2^3}+\cdots\frac{1}{2^{k-2}}\right)\sum_{\substack{d|2^km\\ d ~\text{odd}}} \frac{1}{d}\\
&=\left(2-\frac{3}{2^k}\right)\sum_{\substack{d|2^km\\ d ~\text{odd}}}\frac{1}{d}.
\end{aligned}
\end{equation}

\noindent
Again, the right-hand side of \eqref{triequiv} is
\begin{equation}\label{triequivright}\begin{aligned}&\sum_{d|2^km} \frac{1+2(-1)^d}{d} \\
&=\sum_{\substack{d|2^km\\ d ~\text{odd}}} \frac{1+2(-1)^{d}}{d} +\sum_{\substack{d|2^km\\ d =2d_1,~ d_1~\text{odd}}} \frac{1+2(-1)^{d}}{d}+ \sum_{\substack{d|2^km\\ d=2^2d_1,~ d_1~\text{odd}}} \frac{1+2(-1)^{d}}{d}\\
&\quad+\sum_{\substack{d|2^km\\ d =2^3d_1,~ d_1~\text{odd}}} \frac{1+2(-1)^{d}}{d}+\cdots+\sum_{\substack{d|2^km\\ d =2^kd_1,~ d_1~\text{odd}}} \frac{1+2(-1)^{d}}{d}\\
&=-\sum_{\substack{d|2^km\\ d ~\text{odd}}} \frac{1}{d}+\frac{3}{2}\left(1+\frac{1}{2}+\frac{1}{2^2}+\frac{1}{2^3}+\cdots\frac{1}{2^{k-1}}\right)\sum_{\substack{d|2^km\\ d ~\text{odd}}} \frac{1}{d}\\&=\left(2-\frac{3}{2^k}\right)\sum_{\substack{d|2^km\\ d ~\text{odd}}}\frac{1}{d}.
\end{aligned}
\end{equation}
From \eqref{triequivleft} and \eqref{triequivright}, we arrive at \eqref{triequiv} for this case.
\end{proof}

\noindent
\begin{corollary}
For any positive integer $n$, we have
$$\sum_{\substack{d|n \\ \frac{n}{d}\equiv1~\text{or}~2~(\textup{mod}~3)}} \frac{(-1)^d}{d} + \sum_{\substack{d|n \\ \frac{n}{d}\equiv0~(\textup{mod}~3)}} \frac{1}{d} = \sum_{j=1}^{n} \frac{(-1)^j}{j} \binom{n}{j} t_{5,j}(n).$$
\end{corollary}
\begin{proof}
The result follows by setting $s=5$ in Theorem \ref{theorem1}.
\end{proof}

\begin{corollary}
Let $n$ be a positive integer and $\sigma(n)$ denote the sum of positive divisors of $n$. Let $p$ be an odd prime such that $p\mid n$ and $p^2 \nmid n$. If $\frac{n}{p} \equiv 1 \text{ or } p-1 (\textup{mod}~p)$, then 
\begin{align*} \sum_{j=1}^{n} \frac{(-1)^j}{j} \binom{n}{j} t_{p+2,j}(n)=\frac{\sigma(n)}{n}-\frac{2}{p}. \\\intertext{Otherwise,}
\sum_{j=1}^{n} \frac{(-1)^j}{j} \binom{n}{j} t_{p+2,j}(n)=\frac{\sigma(n)}{n}-\frac{1}{p}.\end{align*}
\end{corollary}
\begin{proof}
Let $p$ be as stated in the corollary. Setting $s=p+2$ in Theorem \ref{theorem1}, it follows that, 
if $\frac{n}{p} \equiv 1 \text{ or } p-1 (\textup{mod}~p)$, then 
\begin{align*} \sum_{j=1}^{n} \frac{(-1)^j}{j} \binom{n}{j} t_{p+2,j}(n)=-\frac{1}{p} + \sum_{\substack{d|n \\ d \neq p}} \frac{1}{d}. \\\intertext{Otherwise,}
\sum_{j=1}^{n} \frac{(-1)^j}{j} \binom{n}{j} t_{p+2,j}(n)=\sum_{\substack{d|n \\ d \neq p}} \frac{1}{d}.\end{align*}
\end{proof}
As $$\sum_{\substack{d|n \\ d \neq p}} \frac{1}{d} =\sum_{\substack{d|n }} \frac{1}{d} - \frac{1}{p}=\frac{1}{n}\sum_{\substack{d|n }} \frac{n}{d}-\frac{1}{p}=\frac{1}{n}\sum_{\substack{d|n }}d-\frac{1}{p}=\frac{\sigma(n)}{n}-\frac{1}{p},$$
we readily arrive at the desired results.

\section*{\textbf{Acknowledgement}}
The first  author was partially supported by an ISPIRE Fellowship for Doctoral Research, DST, Government of India. The author acknowledges the funding agency.

\end{document}